\documentclass[a4paper,12pt]{article}
\usepackage{amssymb}
\usepackage{amsthm}
\usepackage{graphicx}
\usepackage{amsmath}
\usepackage{fancyhdr}
\usepackage[all,cmtip]{xy}
\usepackage{stmaryrd}
\usepackage{comment} 
\usepackage[hidelinks]{hyperref}

\theoremstyle{plain}
\newtheorem{theorem}{Theorem}[section]
\newtheorem{proposition}[theorem]{Proposition}
\newtheorem{lemma}[theorem]{Lemma}
\newtheorem{corollary}[theorem]{Corollary}
\theoremstyle{remark}

\newtheorem{remark}[theorem]{\bf Remark}
\newtheorem{example}[theorem]{\bf Example}

\newtheorem*{acknowledgements}{\sc Acknowledgements}

\renewcommand{\baselinestretch}{1.08}

\newenvironment{reference}{
\begin{flushleft}\normalsize{\textsc{References}}\end{flushleft}
\begingroup
\renewcommand{\section}[2]{}
\begin{thebibliography}{99}
\setlength{\itemsep}{-5pt}
\small
}{
\end{thebibliography}
\endgroup
}

\makeatletter
\@addtoreset{equation}{section}
\makeatother

\makeatletter
\renewcommand{\section}{
\@startsection{section}{1}{\z@}
{3.5ex \@plus -1ex \@minus -.2ex}
{2.3ex \@plus.2ex}
{\reset@font\normalsize\scshape}}

\makeatother

\makeatletter
\renewcommand{\subsection}{
\@startsection{subsection}{2}{\z@}
{-3.5ex \@plus -1ex \@minus -.2ex}
{-2.3ex \@plus.2ex}
{\reset@font\normalsize\scshape}}
\makeatother

\usepackage[OT2,T1]{fontenc}
\DeclareSymbolFont{cyrletters}{OT2}{wncyr}{m}{n}
\DeclareMathSymbol{\Sha}{\mathalpha}{cyrletters}{"58}

\title{\vspace*{-22mm}
\large{\textbf{
On $\mathbb{Z}_2$-extensions of real quadratic fields with class group of $2$-rank three
\\
}}
\footnotetext{Department of Mathematics, 
Rikkyo University, 
3-34-1, Nishi-Ikebukuro, 
Toshima-ku, 
Tokyo 171-8501, Japan. 
\texttt{ysaito.rmath@gmail.com}
}
}

\author{
\textsc{\normalsize Yuito Saito}
}

\date{}

\begin{document}
{\renewcommand{\baselinestretch}{1.05} \maketitle}

\vspace*{-11mm}
\renewcommand{\abstractname}{}
{\renewcommand{\baselinestretch}{1.05}
\begin{abstract}{\small 
\noindent
\textsc{Abstract.}  
In this paper, we study the unramified Iwasawa module over the cyclotomic $\mathbb{Z}_2$-extension of the real quadratic field $\mathbb{Q}(\sqrt{p_1p_2p_3p_4})$, where $p_1, p_2, p_3$, and $p_4$ are distinct odd prime numbers. 
We give a criterion for the finiteness of an unramified Iwasawa module of a number field. 
The criterion is based on the nonexistence of a $2$-group with a certain prescribed quotient.
Using this criterion, we construct an infinite family of such real quadratic fields with unramified Iwasawa module of type $\mathbb{Z}/4\mathbb{Z} \oplus \mathbb{Z}/2\mathbb{Z} \oplus \mathbb{Z}/2\mathbb{Z}$ and $2$-class group of type $\mathbb{Z}/2\mathbb{Z} \oplus \mathbb{Z}/2\mathbb{Z} \oplus \mathbb{Z}/2\mathbb{Z}$.
This gives the first example of an infinite family of real quadratic fields whose ideal class group has $2$-rank three and whose cyclotomic $\mathbb{Z}_2$-extension is totally ramified and satisfies Greenberg's conjecture.
}
\end{abstract}}

\section{Introduction}

Let $l$ be a fixed prime number, and let $\mathbb{Z}_l$ denote the ring of $l$-adic integers. 
Let $k$ be a finite extension of the rational number field $\mathbb{Q}$, and let $k_{\infty}$ be the cyclotomic $\mathbb{Z}_l$-extension of $k$. 
For each integer $n \ge 0$, $k_\infty$ contains a unique intermediate field $k_n$ which is a cyclic extension of degree $l^n$ over $k$. 
Let $L(k_n)$ be the Hilbert $l$-class field of $k_n$. 
For each $n \ge 0$, we denote by $A(k_n)$ the $l$-Sylow subgroup of the ideal class group of $k_n$. 
Then the Galois group $\mathrm{Gal}(L(k_n)/k_n)$ is isomorphic to $A(k_n)$ via the Artin map. 
We define the unramified Iwasawa module $X(k_\infty) = \mathrm{Gal}(L(k_{\infty})/k_{\infty})$, where $L(k_\infty)$ is the maximal unramified abelian pro-$l$-extension of $k_\infty$.
It is well known that $X(k_\infty)$ is a finitely generated torsion $\Lambda$-module, where $\Lambda$ is the formal power series ring $\mathbb{Z}_l[[T]]$, and that $X(k_\infty) \simeq \varprojlim A(k_n)$, where the projective limit is defined with respect to the norm maps. 
By Iwasawa's class number formula (cf.\ \cite{Iwa59, Was}), there exist integers $\lambda_l(k), \mu_l(k) \ge 0$ and $\nu_l(k)$ such that for all sufficiently large $n$, the order $|A(k_n)|$ of $A(k_n)$ is given by
\[
|A(k_n)| = l^{\lambda_l(k) n + \mu_l(k)l^n + \nu_l(k)}.
\]
These integers $\lambda_l(k), \mu_l(k)$, and $\nu_l(k)$ are called the Iwasawa $\lambda$-, $\mu$-, and $\nu$-invariants of the extension $k_\infty/k$, respectively.
Greenberg's conjecture asserts that the unramified Iwasawa module $X(k_\infty)$ is finite, i.e., $\lambda_l(k) = \mu_l(k) = 0$ for any totally real number field $k$ and any prime number $l$ (cf.\ \cite{Gre76}). 
Ferrero and Washington proved that $\mu_l(k) = 0$ for any abelian number field $k$ (cf.\ \cite{FW79}).
It follows that Greenberg's conjecture holds for such fields if and only if $\lambda_l(k) = 0$.
Greenberg's conjecture has been extensively studied from various aspects (cf.\ e.g.\ \cite{BM24,FKOT16,Kum20,Kum25,Oza09,Pag22}).
For the cyclotomic $\mathbb{Z}_2$-extension of real quadratic fields, Ozaki and Taya gave the first nontrivial infinite family for which $X(k_\infty)$ is finite (cf.\ \cite{OT97}).
Since then, various infinite families satisfying Greenberg's conjecture for $l=2$ have been constructed (cf.\ \cite{AM25, JA25, FK05, LS24, Miz04, Miz10, Miz26, MM27, M23, MM11, Nis06, OT97, Yam00}).
Recently, explicit structures of large $X(k_\infty)$ have been completely determined for some of these families. 
Furthermore, Mouhib \cite{M23} and Assarrar--Mouhib \cite{AM25} constructed infinite families for which the $2$-rank of the unramified Iwasawa module $X(k_\infty)$ can be arbitrarily large.
However, all of these known results have one property in common.
Suppose that the cyclotomic $\mathbb{Z}_2$-extension $k_{\infty}/k$ is totally ramified.
In every previously constructed infinite family of real quadratic fields satisfying Greenberg's conjecture for $k_{\infty}/k$, the $2$-rank of the ideal class group of the base quadratic field $k$ is at most two. 
Motivated by these studies, we construct the first explicit infinite family of real quadratic fields $k$ such that their ideal class groups have $2$-rank three, and Greenberg's conjecture holds true for $k_{\infty}/k$.
We also determine the exact structures of their unramified Iwasawa modules.

Our first main theorem is the following.

\begin{theorem}\label{MT1}
Let $p_1, p_2, p_3$ and $p_4$ be distinct prime numbers such that 
\[
p_1 \equiv p_2 \equiv 3 \pmod{8}, \quad
p_3 \equiv 5 \pmod{8}, \quad 
p_4 \equiv 15 \pmod{16}, \quad 
\]
\[
\left(\frac{p_1}{p_2} \right) =1,\ 
\left(\frac{p_1}{p_3} \right) =-1,\ 
\left(\frac{p_1}{p_4} \right) =-1,\ 
\left(\frac{p_2}{p_3} \right) =1 , 
\]
where $\big(\frac{*}{*}\big)$ is the Legendre symbol. 
Let $k= \mathbb{Q}(\sqrt{p_1p_2p_3p_4})$ or $\mathbb{Q}(\sqrt{2p_1p_2p_3p_4})$. 
Then we have
\[
X(k_\infty) \simeq \mathbb{Z}/4\mathbb{Z} \oplus \mathbb{Z}/2\mathbb{Z} \oplus \mathbb{Z}/2\mathbb{Z} \quad \text{and} \quad A(k) \simeq \mathbb{Z}/2\mathbb{Z} \oplus \mathbb{Z}/2\mathbb{Z} \oplus \mathbb{Z}/2\mathbb{Z}.
\]
In particular, $\lambda_2(k)=\mu_2(k)=0, \ \nu_2(k)=4$.
\end{theorem}

By Dirichlet's theorem on arithmetic progressions, there exist infinitely many prime numbers $p_1, p_2, p_3$ and $p_4$ satisfying the conditions of Theorem \ref{MT1}. 
Therefore, we obtain the following corollary. 

\begin{corollary}\label{Cor1}
There exists an infinite family of real quadratic fields such that the cyclotomic $\mathbb{Z}_2$-extension is totally ramified, the unramified Iwasawa module is finite, and its ideal class group has $2$-rank exactly three.
\end{corollary}

Note that the first layer of the cyclotomic $\mathbb{Z}_2$-extension of $\mathbb{Q}$ is given by $\mathbb{Q}_1 = \mathbb{Q}(\sqrt{2})$. 
The two real quadratic fields $\mathbb{Q}(\sqrt{p_1p_2p_3p_4})$ and $\mathbb{Q}(\sqrt{2p_1p_2p_3p_4})$ have the same cyclotomic $\mathbb Z_2$-extension.
Consequently, their unramified Iwasawa modules are naturally isomorphic as profinite abelian groups, and their Iwasawa invariants coincide.
Therefore, it suffices to consider the case of $k=\mathbb{Q}(\sqrt{p_1p_2p_3p_4})$.
Recall that the vanishing of the Iwasawa $\lambda$- and $\mu$-invariants is equivalent to the boundedness of the order of the $2$-Sylow subgroup $A(k_n)$ as $n \to \infty$. 

The proof of Theorem \ref{MT1} consists of two main steps.
First, we give a criterion for the finiteness of an unramified Iwasawa module of a number field (Theorem \ref{MT2}). 
The proof is based on an analysis of the structure of the Galois group $\mathrm{Gal}(L^{S}(k_1)/k)$ (defined later) and on a version of Fukuda's theorem (cf.\ \cite[Theorems 1 and 2]{Fuk94}).
The argument is purely group-theoretic and applies to general number fields satisfying the conditions of Theorem \ref{MT2}.
Using the validity of Leopoldt's conjecture for abelian number fields (cf.\ \cite[Theorem 2]{Bru67}), we obtain, as a corollary of this criterion, a criterion for Greenberg's conjecture for $l=2$ and real quadratic fields (Corollary \ref{MT3}).
Second, we show that the real quadratic field $k$ defined in Theorem \ref{MT1} satisfies the conditions of Corollary \ref{MT3}.
For this purpose, we determine the structures of certain quotients of the $2$-class groups of $k$ and  the first layer $k_1$ of its cyclotomic $\mathbb{Z}_2$-extension.
To deduce these structures, we use the following tools introduced in the next section: the genus formula (cf.\ \cite[Theorem 1]{Lem13}), the theorem of R\'{e}dei and Reichardt (cf.\ \cite[\S3, Satz]{RR}), Kuroda's class number formula (cf.\ \cite{Kur}), and Kubota's theorem (cf.\ \cite{Kub}).
Finally, applying Corollary \ref{MT3} completes the proof of Theorem \ref{MT1}.

\section{Preparations}
\subsection{Ranks of class groups.}

First, we state a brief remark on notation.
For any finite set $U$, we denote its cardinality by $|U|$. 
For an abelian group $A$, let $2A=A^2$ (resp.\ $4A=A^4$) be the subgroup of $A$ generated by the squares (resp.\ fourth powers) of elements of $A$.
Let $\mathrm{rank}_2(A)$ (resp.\ $\mathrm{rank}_4(A)$) be the $2$-rank (resp.\ the $4$-rank) of $A$, which is defined as the $\mathbb{Z}/2\mathbb{Z}$-dimension of $A/2A$ (resp.\ $2A /4A$).
For a number field $K$, we denote the $2$-Sylow subgroup of its ideal class group by $A(K)$ and its unit group by $E(K)$.
For an extension of number fields $K/F$, we denote its Galois group by $\mathrm{Gal}(K/F)$. 
If $K/F$ is finite, we denote its degree by $[K:F]$.
Let $A(K)^{\mathrm{Gal}(K/F)}$ be the subgroup of $A(K)$ generated by the ideal classes fixed by the action of $\mathrm{Gal}(K/F)$. 
Furthermore, the norm map from $K$ to $F$ is denoted by $N_{K/F}$.

In this section, we state some tools needed in the proofs of Theorems \ref{MT1} and \ref{MT2}.
Let $K/F$ be a quadratic extension of number fields. 
We denote by $t(K/F)$ the number of primes (finite and infinite) of $F$ which are ramified in $K/F$. 
To determine the $2$-rank of the ideal class group of the initial layer $k$ and the first layer $k_1$ of the cyclotomic $\mathbb{Z}_2$-extension in Theorem \ref{MT1}, we use the following proposition.

\begin{proposition}[{Genus formula \cite[Theorem 1]{Lem13}}] \label{GF}
In the above setting, the following equality holds:
\[
|A(K)^{\mathrm{Gal}(K/F)}| = \frac{|A(F)| \cdot 2^{t(K/F)-1}}{(E(F) : E(F) \cap N_{K/F} (K^{\times}))}.
\]
\end{proposition}

\noindent 
Suppose that $\mathrm{Gal}(K/F) = \langle \sigma \rangle$, the cyclic group generated by $\sigma$.
If $|A(F)| = 1$, then the ideal $\mathfrak{a}^{1+\sigma}$ is a principal ideal for any ideal class $[\mathfrak{a}] \in A(K)$. 
Consequently, $\sigma$ acts on $A(K)$ simply as $-1$, which implies $A(K)^{1+\sigma} \simeq 0$.
Therefore, $A(K)^{\sigma -1} = A(K)^2$.
This implies $|A(K)^{\mathrm{Gal}(K/F)}| = |A(K)/2A(K)|$.

From here on, we focus on the case where $K$ is a real quadratic field, so that $F = \mathbb{Q}$. 
To determine the $4$-rank of the narrow ideal class group of $K$, whose $2$-Sylow subgroup is denoted by $A^+(K)$, we recall some standard facts and notation. 
Let $D_K$ denote the discriminant of $K$. 
We can decompose $D_K$ into a product of prime discriminants $D_K = \pm 2^e p_1^* \cdots p_t^*$, where $e \in \{0,2, 3\}$ and each $p_i^* = \pm p_i \equiv 1 \pmod{4}$ corresponds to an odd prime factor $p_i$ of $D_K$. 
The narrow genus field $K_G^+$ of $K$ is constructed as $K_G^+ = \mathbb{Q}(\sqrt{\delta}, \sqrt{p_1^*}, \dots, \sqrt{p_t^*})$, where $\delta \in \{\pm 1, \pm 2\}$ and $\delta = 1$ whenever $e=0$. 
Then the genus field $K_G$ is defined as the maximal totally real subfield of $K_G^+$. 
The genus field $K_G$ (resp.\ $K_G^+$) is the maximal abelian extension of $\mathbb{Q}$ contained in the Hilbert $2$-class field (resp.\ the narrow $2$-class field) of $K$.
The nontrivial automorphism $\sigma \in \mathrm{Gal}(K/\mathbb{Q})$ acts trivially on $\mathrm{Gal}(K_G/K)$. 
By class field theory, $\mathrm{Gal}(K_G/K)$ is isomorphic to $A(K)/A(K)^{\sigma-1}$, which is the maximal quotient of $A(K)$ on which $\sigma$ acts trivially.
This implies $\mathrm{Gal}(K_G/K) \simeq A(K)/2A(K)$. 
Similarly, we see that $\mathrm{Gal}(K_G^+/K) \simeq A^+(K)/2A^+(K)$.

To state the next proposition, we introduce two finite sets, $S_1(K)$ and $S_2(K)$. 
Let $S_1(K)$ be the set consisting of sets of two integers $\{D_1, D_2\}$ such that $D_K = D_1D_2$ and $D_i \equiv 0 \text{ or } 1 \pmod{4}$ for $i = 1, 2$. 
We define $S_2(K)$ as the subset of $S_1(K)$ consisting of those sets $\{D_1, D_2\}$ which satisfy $\chi_{D_1}(p) = 1$ for all prime divisors $p$ of $D_2$, and $\chi_{D_2}(p) = 1$ for all prime divisors $p$ of $D_1$.
Here, $\chi_{D_i}(*)$ is the Kronecker symbol for a discriminant $D_i$.

\begin{proposition}[{R\'edei--Reichardt \cite[\S3, Satz]{RR}}] \label{RRT}
In the above setting, the following equalities hold:
\begin{align*}
|S_1(K)| &= |A^{+}(K)/2A^{+}(K)| = 2^{\mathrm{rank}_2(A^{+}(K))} , \\
|S_2(K)| &= |2A^{+}(K)/4A^{+}(K)| = 2^{\mathrm{rank}_4(A^{+}(K))} . 
\end{align*}
\end{proposition}

\subsection{Fundamental units and class number formulas.}

Next, we introduce a useful tool for the class number and fundamental units of real biquadratic extensions of $\mathbb{Q}$. 
It is easy to see that the first layer $k_1 = k(\sqrt{2})$ of the cyclotomic $\mathbb{Z}_2$-extension of the base field $k$ in Theorem \ref{MT1} is a biquadratic extension of $\mathbb{Q}$. 
To determine the order of $A(k_1)$, we use the following proposition, which combines Kuroda's class number formula (cf.\ \cite{Kur}) with Kubota's theorem (cf.\ \cite{Kub}).

\begin{proposition}[{Kuroda \cite{Kur}, Kubota \cite{Kub}}]\label{KKT}
Let $K$ be a real biquadratic extension of $\mathbb{Q}$ and let $F_i \ (i=1,2,3) $ be the three quadratic intermediate fields of $K/\mathbb{Q}$. 
Let $\varepsilon_i$ be a fundamental unit of $F_i$. 
We denote by $h(K)$ and $h(F_i)$ the class numbers of $K$ and $F_i$, respectively. 
Put $Q(K) = (E(K) : \langle -1, \varepsilon_1, \varepsilon_2, \varepsilon_3 \rangle)$. 
Then we have
\[
h(K) = \frac{1}{4} \cdot Q(K)\cdot h(F_1)\cdot h(F_2) \cdot h(F_3).
\]
Furthermore, $Q(K) \in \{1,2,4\}$, and a system of fundamental units of $K$ is given by exactly one of the following forms:
\begin{align*}
\textup{(1)} &\quad \{ \varepsilon_1, \varepsilon_2, \varepsilon_3 \} \\
\textup{(2)} &\quad \{ \sqrt{\varepsilon_1}, \varepsilon_2, \varepsilon_3 \} \qquad (N(\varepsilon_1)=1)\\
\textup{(3)} &\quad \{ \sqrt{\varepsilon_1}, \sqrt{\varepsilon_2}, \varepsilon_3 \} \qquad (N(\varepsilon_1)=N(\varepsilon_2)=1)\\
\textup{(4)} &\quad \{ \sqrt{\varepsilon_1\varepsilon_2}, \varepsilon_2, \varepsilon_3 \} \qquad (N(\varepsilon_1)=N(\varepsilon_2)=1) \\
\textup{(5)} &\quad \{ \sqrt{\varepsilon_1\varepsilon_3}, \sqrt{\varepsilon_2}, \varepsilon_3 \} \qquad (N(\varepsilon_1)=N(\varepsilon_2)=N(\varepsilon_3)=1)\\
\textup{(6)} &\quad \{ \sqrt{\varepsilon_1\varepsilon_2}, \sqrt{\varepsilon_2\varepsilon_3}, \sqrt{\varepsilon_3\varepsilon_1} \} \qquad (N(\varepsilon_1)=N(\varepsilon_2)=N(\varepsilon_3)=1) \\
\textup{(7)} &\quad \{ \sqrt{\varepsilon_1\varepsilon_2\varepsilon_3}, \varepsilon_2, \varepsilon_3 \} \qquad (N(\varepsilon_1)=N(\varepsilon_2)=N(\varepsilon_3)=\pm 1)
\end{align*}
where $N(\varepsilon_i)$ denotes the absolute norm $N_{F_i/\mathbb{Q}}(\varepsilon_i)$.
\end{proposition}

\subsection{Boundedness of class groups.}

Finally, we introduce some notation for $\mathbb Z_l$-extensions. 
For a fixed prime number $l$, let $k_{\infty}$ be a $\mathbb{Z}_l$-extension (not necessarily cyclotomic) of a number field $k$, and let $k_n$ be its $n$-th layer. 
For a finite abelian $l$-group $A$, let $lA=A^l$ and define $\mathrm{rank}_l(A) = \dim_{\mathbb Z/l\mathbb Z}(A/lA)$.
Let $\Sigma$ denote the set of all primes of $k$ which are ramified in $k_\infty/k$.
In particular, every prime in $\Sigma$ lies above $l$.
Let $S$ be a subset of $\Sigma$. 
Let $L^{S}(k_n)$ be the maximal $S$-decomposed unramified abelian $l$-extension of $k_n$, that is, the maximal unramified abelian $l$-extension of $k_n$ in which all prime ideals of $k_n$ lying above $S$ split completely.
Let $D_{S}(k_n)$ be the subgroup of $A(k_n)$ generated by the ideal classes containing a product of prime ideals of $k_n$ lying above $S$. 
Let $A^{S}(k_n)$ be the $l$-Sylow subgroup of the $S$-ideal class group of $k_n$, that is, $A^{S}(k_n) = A(k_n)/D_{S}(k_n)$.
Then we have the isomorphism $\mathrm{Gal}(L^{S}(k_n)/k_n) \simeq A^{S}(k_n)$ via the Artin map.
We note that, in this paper, we consider only the cases in which $S= \emptyset$ or $S= \Sigma$. 
In particular, if $S = \emptyset$, then $L^{\emptyset}(k_n)$ and $A^{\emptyset}(k_n)$ are the usual Hilbert $l$-class field $L(k_n)$ and $l$-class group $A(k_n)$, respectively.

We assume that every prime ideal in $\Sigma$ is totally ramified in $k_{\infty}/k$.
Let $\gamma$ be a fixed topological generator of $\Gamma =\mathrm{Gal}(k_{\infty}/k)$. 
We identify the complete group ring $\mathbb{Z}_l[[\Gamma]]$ with the formal power series ring $\Lambda = \mathbb{Z}_l[[T]]$ by the correspondence $\gamma \leftrightarrow 1+T$. 
We define the $S$-decomposed unramified Iwasawa module $X^{S} = X^{S}(k_\infty)=\mathrm{Gal}(L^{S}(k_{\infty})/k_{\infty})$, where $L^{S}(k_\infty)$ is the maximal $S$-decomposed unramified abelian pro-$l$-extension of $k_\infty$, that is, the maximal unramified abelian pro-$l$-extension of $k_{\infty}$ in which all primes of $k_{\infty}$ lying above $S$ split completely.
It is well known that $X^{S}$ is a finitely generated torsion $\Lambda$-module and that $X^{S}(k_\infty)\simeq\varprojlim A^{S}(k_n)$, where the projective limit is defined with respect to the norm maps.
In particular, $X^{\emptyset}(k_\infty)= X(k_\infty)$.
Furthermore, there exists a $\Lambda$-submodule $Y^{S}$ of $X^{S}$ such that $X^{S}/(\omega_n/T)Y^{S}  \simeq A^{S}(k_n)$ for all $n \ge 0$, where $\omega_n = (1+T)^{l^n}-1$ (cf.\ \cite{Iwa73, Was}).

We state the following variant of Fukuda's theorem in a more general form (cf.\ \cite[Theorems 1 and 2]{Fuk94}).

\begin{proposition}\label{FT1}
Let $k_{\infty}/k$ be a $\mathbb{Z}_{l}$-extension of a number field $k$ and let $S$ be a subset of $\Sigma$. 
Then the following assertions hold:
\begin{enumerate}
\item[\textup{(1)}] 
Let $n_0 \ge 0$ be an integer such that every prime ideal of $k_{n_0}$ lying above $\Sigma$ is totally ramified in $k_\infty/k_{n_0}$.
If there exists an integer $n \ge n_0$ such that $|A^{S}(k_{n+1})| = |A^{S}(k_{n})|$, then $A^{S}(k_m) \simeq A^{S}(k_n)$ for all $m \ge n$.

\item[\textup{(2)}] 
Let $n_0 \ge 0$ be an integer such that every prime ideal of $k_{n_0}$ lying above $\Sigma$ is totally ramified in $k_\infty/k_{n_0}$.
If there exists an integer $n \ge n_0$ such that $\mathrm{rank}_l ( A^{S}(k_{n+1}) )=\mathrm{rank}_l ( A^{S}(k_{n}) )$, then $\mathrm{rank}_l ( A^{S}(k_m) ) =\mathrm{rank}_l ( A^{S}(k_n) )$ for all $m \ge n$.

\item[\textup{(3)}] 
Assume that the set $\Sigma$ consists of a unique prime ideal which is totally ramified in $k_\infty/k$.
If the lifting map $\iota_{n} : A^{S}(k) \to A^{S}(k_n)$ is the zero map for some $n \ge 1$, then $A^{S}(k_m) \simeq A^{S}(k_n)$ for all $m \ge n$.
\end{enumerate}
\end{proposition}

\begin{proof}
(1) By replacing the base field if necessary, we may assume that $n=n_0=0$. 
If $|A^{S}(k_1)| = |A^{S}(k)|$, then we have $|X^{S}/(\omega_1/T)Y^{S}| = |X^{S}/Y^{S}|$, which implies $Y^{S} = (\omega_1/T)Y^{S}$. 
Since $\omega _1/T \in (l,T)$ and $Y^{S}$ is a finitely generated $\Lambda$-module, Nakayama's lemma implies that $Y^{S}=0$. 
Therefore, $A^{S}(k_m) \simeq X^{S} \simeq A^{S}(k)$ for all $m \ge 0$.

(2) Similarly to (1), we may assume that $n=n_0=0$. 
If $\mathrm{rank}_l (A^{S}(k_1)) = \mathrm{rank}_l (A^{S}(k))$, then we have $|A^{S}(k_1)/lA^{S}(k_1)| = |A^{S}(k)/lA^{S}(k)|$.
This implies 

\noindent $|X^{S}/((\omega_1/T)Y^{S} + lX^{S})| = |X^{S}/(Y^{S}+lX^{S})|$, which yields $Y^{S}+lX^{S} =(\omega_1/T)Y^{S} + lX^{S}$. 
Since $\omega_1/T \in (l, T)$ and $(Y^{S}+lX^{S})/lX^{S}$ is a finitely generated $\Lambda$-module, Nakayama's lemma implies that $Y^{S} \subset lX^{S}$, and hence we have $(\omega_n/T)Y^{S} \subset lX^{S}$.
Thus, $A^{S}(k_m)/lA^{S}(k_m) \simeq X^{S}/lX^{S} \simeq A^{S}(k)/lA^{S}(k)$ for all $m \ge 0$. 
Therefore, $\mathrm{rank}_l (A^{S}(k_m)) = \mathrm{rank}_l (A^{S}(k))$ for all $m \ge 0$.

(3) By assumption, we have $A^{S}(k) \simeq \mathrm{Gal}(L^{S}(k)k_{\infty}/k_{\infty}) \simeq X^{S}/TX^{S}$, so $Y^{S}=TX^{S}$. 
Hence, $A^{S}(k_m) \simeq X^{S}/\omega_m X^{S}$ for all $m \ge 0$. 
If the lifting map $\iota_{n} : A^{S}(k) \to A^{S}(k_n)$ is the zero map, then the homomorphism $X^{S}/TX^{S} \to X^{S}/\omega_n X^{S}$ given by $x' \pmod{TX^{S}} \mapsto (\omega_n /T)x' \pmod{\omega_n X^{S}}$ is also the zero map (\cite[Theorem 8]{Iwa73}). 
This implies that $(\omega_n/T)X^{S} \subset \omega_n X^{S} = T(\omega_n/T)X^{S}$. 
Since $T \in (l,T)$ and $(\omega_n/T)X^{S}$ is a finitely generated $\Lambda$-module, Nakayama's lemma implies that $(\omega_n/T)X^{S}=0$. 
This implies $\omega_m X^{S} =0$ for all $m \ge n$. 
Therefore, $A^{S}(k_m) \simeq X^{S} \simeq A^{S}(k_n)$ for all $m \ge n$.
\end{proof}

From here on, let $k_{\infty}/k$ be the cyclotomic $\mathbb{Z}_l$-extension of a number field $k$.
For each $n \ge 0$, we put $B(k_n) = \{a \in A(k_n) \mid a^\gamma = a\}$, which is a subgroup of $A(k_n)$. 
We briefly recall Leopoldt's conjecture (cf.\ \cite[\S 5.5]{Was}). 
In the case of a totally real number field $k$ and a prime $l$, this conjecture asserts that the cyclotomic $\mathbb Z _l$-extension $k_{\infty}$ is the unique $\mathbb Z _l$-extension of $k$. 

\begin{proposition}[{Greenberg \cite[Proposition 1]{Gre76}}] \label{GT}
Suppose that Leopoldt's conjecture holds for a totally real number field $k$ and a prime number $l$. Then $|B(k_n)|$ is bounded as $n \rightarrow \infty$.
\end{proposition}

It is well known that Leopoldt's conjecture holds for any abelian extension of $\mathbb{Q}$ (cf.\ \cite[Theorem 2]{Bru67}). 
Consequently, the assumption of Proposition \ref{GT} is satisfied for real quadratic fields. 
Under the assumption that no prime ideal of $k$ lying above $l$ splits in $k_1/k$, we have $D_{\Sigma}(k_n) \subset B(k_n)$, which implies that $|D_{\Sigma}(k_n)|$ is bounded as $n \to \infty$. 
The boundedness of both $|D_{\Sigma}(k_n)|$ and $|A^{\Sigma}(k_n)|$ implies the boundedness of $|A(k_n)|$.
This fact is essential for the proof of Corollary \ref{MT3}.

\section{Criteria}
In this section, we suppose that $l=2$. 
Recall that $\Sigma$ denotes the set of all prime ideals of $k$ which are ramified in $k_\infty/k$.
In the present setting, every prime ideal in $\Sigma$ lies above $2$.
Note that we consider only the cases in which $S= \emptyset$ or $S= \Sigma$. 

We give a criterion for the finiteness of the $S$-decomposed unramified Iwasawa module of a general number field (Theorem \ref{MT2}). 
Our second main theorem is the following.

\begin{theorem}\label{MT2}
Let $k_{\infty}/k$ be a $\mathbb{Z}_2$-extension of a number field $k$ and let $S$ be either $\emptyset$ or $\Sigma$.
Assume that the following four conditions hold:
\begin{itemize}
\item[\textup{(1)}] 
The set $\Sigma$ consists of a unique prime ideal which is totally ramified in $k_\infty/k$;
\item[\textup{(2)}] $A^{S}(k) \simeq \mathbb{Z}/2\mathbb{Z} \oplus \mathbb{Z}/2\mathbb{Z}$;
\item[\textup{(3)}] $A^{S}(k_1) \simeq \mathbb{Z}/4\mathbb{Z} \oplus \mathbb{Z}/2\mathbb{Z}$;
\item[\textup{(4)}] $\mathrm{rank}_4(A^{S}(k_2)) = 1$.
\end{itemize}
Then we have $X^{S}(k_{\infty}) \simeq \mathbb{Z}/4\mathbb{Z} \oplus \mathbb{Z}/2\mathbb{Z}$. 
\end{theorem}

\begin{proof}
First, we show that $G=\mathrm{Gal}(L^{S}(k_1)/k)$ is a non-abelian group. 
By the maximality of $L^{S}(k_1)$, $L^{S}(k_1)/k$ is a Galois extension. 
By condition (3), we have $|G| = [k_1:k]\,[L^{S}(k_1):k_1] = 16$.
Suppose, for contradiction, that $G$ is abelian.
Let $\mathfrak{l}$ be the unique prime ideal in $\Sigma$.
By condition \textup{(1)}, $\mathfrak l$ is totally ramified in $k_1/k$.
Since $L^{S}(k_1)/k_1$ is unramified, $\mathfrak{l}$ is the unique prime ideal ramified in $L^{S}(k_1)/k$.
Let $M$ be the inertia field of $\mathfrak{l}$ in $L^{S}(k_1)/k$. 
Since $\mathfrak{l}$ is totally ramified in $k_1/k$, the inertia group of $\mathfrak{l}$ in $L^{S}(k_1)/k$ has order $2$.
Then $\mathrm{Gal}(L^{S}(k_1)/M)$ has order $2$, and in particular, we have $[M:k]=8$. 
Moreover, since $G$ is assumed to be abelian, $M/k$ is a Galois extension, and in particular, an abelian extension. 
Since $M$ is the inertia field of $\mathfrak{l}$ in $L^{S}(k_1)/k$, $M/k$ is an unramified extension. 
If $S=\Sigma$, then the unique prime ideal of $k_1$ above $\mathfrak{l}$ splits completely in $L^S(k_1)/k_1$.
Hence the decomposition group of $\mathfrak{l}$ in $L^{S}(k_1)/k$ coincides with its inertia group, so that $\mathfrak l$ splits completely in $M/k$.
If $S=\emptyset$, no decomposition condition is imposed.
Thus, in either case, $M/k$ is a $S$-decomposed unramified abelian $2$-extension of $k$. 
Hence $M\subset L^{S}(k)$.
By condition (2), we obtain $[M:k] \leq [L^{S}(k):k] = |A^{S}(k)| = 4$, which contradicts $[M:k]=8$.
Therefore, $G$ is a non-abelian group. 
Next, we determine the structure of $A^{S}(k_2)$.
The intermediate fields and the generators of the Galois group of $L^{S}(k_2)/k$ are important for determining the structure of $A^{S}(k_2)$ (see Figure \ref{fig1}).
Let $X^{S}(k_i)= \mathrm{Gal}(L^{S}(k_i)/k_i)$ for $i\in\{1,2\}$. 
Recall that $X^{S}(k_i) \simeq A^{S}(k_i)$ via the Artin map.
By conditions (2) and (3), we have $\mathrm{rank}_2(A^{S}(k)) = \mathrm{rank}_2(A^{S}(k_1)) = 2$.
Moreover, by condition (1), we may take $n_0=0$ in Proposition \ref{FT1} (2).
Therefore, we obtain $\mathrm{rank}_2(A^{S} (k_2))=2$.
Together with condition (4), this implies that, for some integer $m\geq 2$, 
\[
X^{S}(k_2) = \langle g,h \mid g^{2^m}=h^2=1,\ hgh^{-1}=g \rangle,
\]
with an element $g$ of order $2^m$ and an element $h$ of order $2$. 
Our goal is to prove that $A^{S}(k_2) \simeq \mathbb{Z}/4\mathbb{Z} \oplus \mathbb{Z}/2\mathbb{Z}$, that is, $X^{S}(k_2) \simeq \mathbb{Z}/4\mathbb{Z} \oplus \mathbb{Z}/2\mathbb{Z}$.
The restriction map $ \mathrm{Gal}(L^{S}(k_2)/k_2)\to \mathrm{Gal}(L^{S}(k_1)/k_1) $ induces a surjective homomorphism $\pi : X^{S}(k_2) \twoheadrightarrow X^{S}(k_1)$.
By condition (3), $|A^{S}(k_1)|=8$ , and hence we have $|\ker \pi | = \frac{|X^{S}(k_2)|}{|X^{S}(k_1)|} = 2^{m-2}$. 
Note that $4X^{S}(k_2)= \langle g^4 \rangle$ has order $2^{m-2}$.
Since $4X^{S}(k_1) \simeq 0$, we see that $4X^{S}(k_2) \subset \ker \pi$.
Comparing their orders, we have $\ker \pi = \langle g^4 \rangle$.
Now, we define elements $x, y \in X^{S}(k_1)$ by $x = \pi(g)$ and $y = \pi(h)$, respectively.
Since $x^4 = 1$, we see that the order of $x$ is either $1,2,$ or $4$. 
If $x^2=1$, then $g^2 \in \ker \pi = \langle g^4\rangle $, which is a contradiction. 
Hence, the order of $x$ is $4$.
Similarly, since $y^2 = 1$, we see that the order of $y$ is either $1$ or $2$.
If $y=1$, then $h \in \ker \pi = \langle g^4\rangle \subset \langle g\rangle$, which contradicts $h \notin \langle g \rangle$. 
Hence, the order of $y$ is $2$. 
Thus, we have 
\[
X^{S}(k_1) = \langle x, y \mid x^4=y^2=1, yxy^{-1}=x \rangle.
\]
We now determine the action of $\mathrm{Gal}(k_1/k)$ on $X^{S}(k_1)$. 
Since $L^{S}(k)/k$ is an unramified extension and $\mathfrak l$ is totally ramified in $k_1/k$, we have $L^{S}(k)\cap k_1=k$.
Moreover, $ L^{S}(k)k_1 \subset L^{S}(k_1)$.
By condition \textup{(2)}, $[L^{S}(k)k_1:k_1] =[L^{S}(k):k]=4$.
Put $H=\mathrm{Gal}( L^{S}(k_1)/L^{S}(k)k_1)$.
Then $H$ has order $2$, and $X^{S}(k_1)/H \simeq \mathrm{Gal}(L^{S}(k)k_1/k_1) \simeq A^{S}(k) \simeq \mathbb{Z}/2\mathbb{Z} \oplus \mathbb{Z}/2\mathbb{Z}$.
Among the subgroups of order $2$ of $X^{S}(k_1) \simeq \mathbb{Z}/4\mathbb{Z} \oplus \mathbb{Z}/2\mathbb{Z}$, the subgroup $\langle x^2\rangle$ is the unique one whose quotient is elementary abelian. 
Therefore, $H=\langle x^2\rangle $.
Since both $L^{S}(k)/k$ and $k_1/k$ are abelian and $L^{S}(k)\cap k_1=k$, the extension $L^{S}(k)k_1/k$ is abelian. 
Hence $G/H \simeq \mathrm{Gal}(L^{S}(k)k_1/k)$ is abelian.
By the maximality of $L^{S}(k_2)$, $L^{S}(k_2)/k$ is a Galois extension. 
Let $\mathrm{Gal}(k_2/k)=\langle\delta\rangle$, and let $\gamma$ be the restriction of $\delta$ to $k_1$, and then $\mathrm{Gal}(k_1/k)=\langle\gamma\rangle$.
Let $\tilde{\delta}$ be the extension of $\delta$ to $L^{S}(k_2)$ which generates the inertia group of a prime lying above $\mathfrak{l}$ in $L^{S}(k_2)/k$ and let $\tilde{\gamma}$ be the restriction of $\tilde{\delta}$ to $L^{S}(k_1)$. 
Since $G/\langle x^2\rangle$ is abelian, we have $\tilde{\gamma} x \tilde{\gamma}^{-1} \in \{x,x^{-1}\}$ and $\tilde{\gamma}y\tilde{\gamma}^{-1} \in \{y,x^2y\}$.
We consider two cases, depending on the values of $\tilde{\gamma} x \tilde{\gamma}^{-1}$ and $\tilde{\gamma} y \tilde{\gamma}^{-1}$.
Since $G$ is non-abelian, the case simultaneously $\tilde{\gamma}x\tilde{\gamma}^{-1}=x$ and $\tilde{\gamma} y \tilde{\gamma}^{-1}=y$ does not occur.

\noindent \textbf{Case 1:} 
We suppose that $\tilde{\gamma} y \tilde{\gamma}^{-1} = x^2y$. 
Suppose, for contradiction, that $m\geq 3$.
By the definition of $x$ and $y$, we have
\[
\pi (\tilde{\delta} h \tilde{\delta}^{-1})  = \tilde{\gamma} \pi(h) \tilde{\gamma}^{-1} = \tilde{\gamma} y \tilde{\gamma}^{-1} = x^2y = \pi (g^2h).
\]
This implies that $\tilde{\delta} h \tilde{\delta}^{-1} (g^2h)^{-1} \in \ker\pi = \langle g^4\rangle$. 
Thus, there exists an integer $t\in \mathbb{Z}$ such that $\tilde{\delta} h \tilde{\delta}^{-1} = g^{2+4t}h$. 
Since the order of $g$ is $2^m$, the order of $g^{2+4t} =g^{2(1+2t)}$ is $2^{m-1}$. 
Thus, the order of $g^{2+4t}h$ is $2^{m-1} \ge 2^{3-1}=4$. 
Since $h$ is an element of order $2$, $\tilde{\delta} h \tilde{\delta}^{-1}$ is also an element of order $2$.
This is a contradiction. 
Hence, we have $m=2$.
Therefore $X^{S}(k_2)\simeq  \mathbb{Z}/4\mathbb{Z} \oplus \mathbb{Z}/2\mathbb{Z}$, which means $A^{S}(k_2)\simeq A^{S}(k_1)$.
Therefore, by Proposition \ref{FT1} (1), $A^{S}(k_n) \simeq \mathbb{Z}/4\mathbb{Z} \oplus \mathbb{Z}/2\mathbb{Z}$ for all $n\ge1$, and hence $X^{S}(k_{\infty}) \simeq \mathbb{Z}/4\mathbb{Z} \oplus \mathbb{Z}/2\mathbb{Z} $. 

\noindent \textbf{Case 2:} 
We suppose  that $\tilde{\gamma} x \tilde{\gamma}^{-1} = x^{-1}$ and $\tilde{\gamma} y \tilde{\gamma}^{-1} = y$. 
For any element $a \in X^{S}(k_1)$, we can write $a = x^ry^s$ with $r,s \in \mathbb{Z}$. 
Since $X^{S}(k_1)$ is abelian, we have $a (\tilde{\gamma} a \tilde{\gamma}^{-1}) = (x^ry^s)(x^{-r}y^s) = y^{2s} = 1$, which implies that $A^{S}(k_1)^{1+\gamma} \simeq 0$.
Note that $A^{S}(k_1)^{1+\gamma} = \iota_{1} ( N_{k_1/k} (A^{S}(k_1)))$.
The norm map $N_{k_1/k}: A(k_1) \to A(k)$ induces a surjective homomorphism  $A^{S}(k_1) \twoheadrightarrow A^{S}(k)$.
Hence, the lifting map $\iota_{1}: A^{S}(k) \to A^{S}(k_1)$ is the zero map. 
By Proposition \ref{FT1} (3), $A^{S}(k_n)\simeq \mathbb{Z}/4\mathbb{Z} \oplus \mathbb{Z}/2\mathbb{Z} $ for all $n \ge 1$, and hence $X^{S}(k_{\infty}) \simeq \mathbb{Z}/4\mathbb{Z} \oplus \mathbb{Z}/2\mathbb{Z} $. 

\noindent This completes the proof. 
\end{proof}

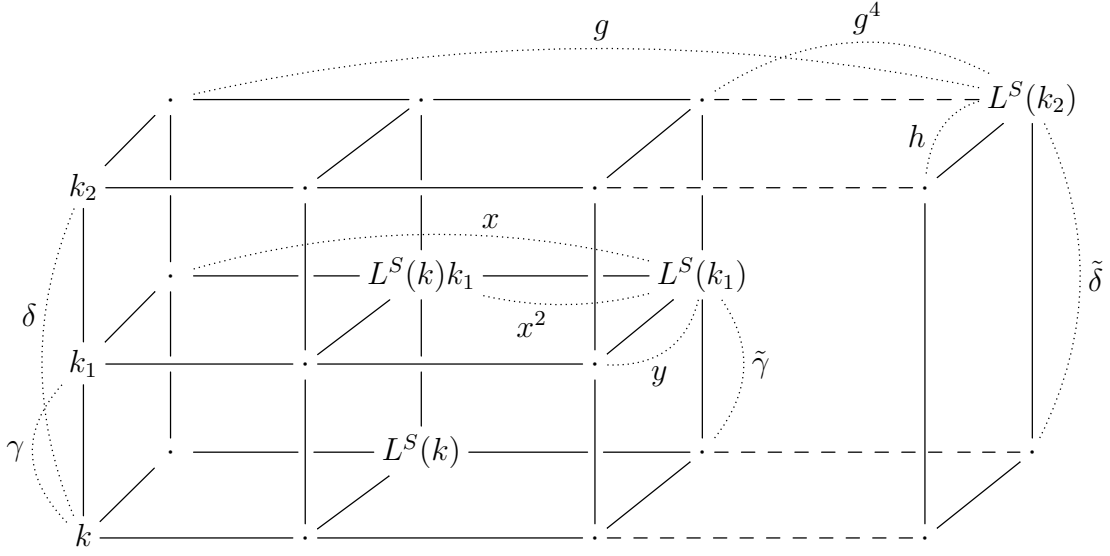
\begin{figure}[htbp] 
\centering
$\entrymodifiers={+!!<0pt,\fontdimen22\textfont2>}
\xymatrix@=1\baselineskip{
& \mbox{$\,\,\cdot\,\,$} \ar@{-}[rrr] \ar@{-}[dd]|\hole & & & \cdot \ar@{-}[rrr] \ar@{-}[dd]|\hole & & & \cdot \ar@{--}[rrrr] \ar@{-}[dd]|\hole & & & & L^{S}(k_2) \ar@{-}[dddd] \ar@/^1.5pc/@{.}[dddd]^{\mbox{$\tilde{\delta} $}} \ar@/_1.8pc/@{.}[llll]_{\mbox{$ g^4 $}} \ar@/_1.2pc/@{.}[ld]_(0.75){\mbox{$ h $}} \ar@/_1.6pc/@{.}[llllllllll]_{\mbox{$ g $}} &&&&\\
k_2 \ar@{-}[rrr] \ar@{-}[ur] \ar@{-}[dd]  \ar@/_1.3pc/@{.}[dddd]_(0.37){\mbox{$ \delta$}}  & & & \cdot \ar@{-}[rrr] \ar@{-}[ur] \ar@{-}[dd] & & & \cdot \ar@{--}[rrrr] \ar@{-}[ur] \ar@{-}[dd] & & & & \cdot \ar@{-}[ur] \ar@{-}[dddd] \\
& \mbox{$\,\,\cdot\,\,$} \ar@{-}[rrr]|(0.54)\hole \ar@{-}[dd]|(0.5)\hole \ar@/^1.3pc/@{.}[rrrrrr]^(0.6){\mbox{$x$}}  & & & L^{S}(k)k_1 \ar@{-}[rrr]|(0.62)\hole  \ar@{-}[dd]|(0.5)\hole   & & & L^{S}(k_1)\ar@{-}[dd] \ar@/^1.3pc/@{.}[dd]^{\mbox{$ \tilde{\gamma} $}} \ar@/^1.1pc/@{.}[ld]^(0.65){\mbox{$ y $}} \ar@/^0.9pc/@{.}[lll]^(0.6){\mbox{$x^2$}}& & & &  \\
k_1 \ar@{-}[rrr]   \ar@{-}[ur]     \ar@{-}[dd] \ar@/_1.6pc/@{.}[dd]_(0.5){\mbox{$ \gamma $}}& & & \cdot \ar@{-}[rrr] \ar@{-}[ur] \ar@{-}[dd]  & & & \cdot  \ar@{-}[ur] \ar@{-}[dd] & & & & \\
& \mbox{$\,\,\cdot\,\,$} \ar@{-}[rrr]|(0.54)\hole & & & L^{S}(k) \ar@{-}[rrr]|(0.62)\hole & & & \cdot \ar@{--}[rrrr]|(0.75)\hole & & & & \cdot \\
k   \ar@{-}[rrr] \ar@{-}[ur]  & & & \cdot \ar@{-}[rrr] \ar@{-}[ur] & & & \cdot \ar@{--}[rrrr] \ar@{-}[ur] & & & & \cdot \ar@{-}[ur]
}$
\caption{The intermediate fields and the Galois groups of $L^{S}(k_2)/k$}
\label{fig1}
\end{figure}

In the setting of the following corollary, every prime ideal of $k$ lying above $2$ is totally ramified in the cyclotomic $\mathbb{Z}_2$-extension $k_\infty/k$. 
We take $S=\Sigma$, so that $S$ consists of all prime ideals of $k$ lying above $2$.
The following corollary of Theorem \ref{MT2} plays an important role in the proof of Theorem \ref{MT1}.

\begin{corollary}\label{MT3}
Let $k=\mathbb{Q}(\sqrt{m})$, where $m>1$ is a positive odd square-free integer, and let
$k_{\infty}/k$ be the cyclotomic $\mathbb{Z}_2$-extension of $k$.
Assume that the following four conditions hold:
\begin{itemize}
\item[\textup{(1)}] $2$ does not split in $k/\mathbb{Q}$;
\item[\textup{(2)}] $A^{\Sigma}(k) \simeq \mathbb{Z}/2\mathbb{Z} \oplus \mathbb{Z}/2\mathbb{Z}$;
\item[\textup{(3)}] $A^{\Sigma}(k_1) \simeq \mathbb{Z}/4\mathbb{Z} \oplus \mathbb{Z}/2\mathbb{Z}$;
\item[\textup{(4)}] $\mathrm{rank}_4(A^{\Sigma}(k_2)) = 1$.
\end{itemize}
Then $X^{\Sigma}(k_{\infty}) \simeq \mathbb{Z}/4\mathbb{Z} \oplus \mathbb{Z}/2\mathbb{Z}$. 
In particular, $\lambda_2(k)=\mu_2(k)=0$.
\end{corollary}

\begin{proof}
Since $m>1$ is odd and $2$ does not split in $k/\mathbb{Q}$, there is
a unique prime ideal of $k$ lying above $2$, and it is totally ramified
in the cyclotomic $\mathbb{Z}_{2}$-extension $k_\infty/k$.
By Theorem \ref{MT2}, we have $X^{\Sigma}(k_{\infty}) \simeq \mathbb{Z}/4\mathbb{Z} \oplus \mathbb{Z}/2\mathbb{Z}$.
Furthermore, by Proposition \ref{GT}, $|D_{\Sigma}(k_n)|$ is bounded as $n\rightarrow \infty$.
Therefore, $|A(k_n)|$ is bounded as $n \to \infty$.
In particular, we have $\lambda_2(k)=\mu_2(k)=0$.
This completes the proof. 
\end{proof}

\begin{remark}
In Corollary \ref{MT3}, since $|D_\Sigma(k_n)|\le 2$ for each $n\ge0$, we have $|X(k_{\infty})|\le 16$. 
Condition (1) in Corollary \ref{MT3} ensures this.
In fact, by condition (1), we consider two cases depending on whether $2$ is inert or ramified in $k/\mathbb{Q}$.
Note that $A(\mathbb{Q}_n) \simeq 0$ for all $n \ge 0$ (cf.\ \cite{Web86}).
If $2$ is inert in $k/\mathbb{Q}$, the prime ideal of $k_n$ lying above $2$ is a principal ideal, and hence $|D_\Sigma(k_n)|= 1$.
If $2$ is ramified in $k/\mathbb{Q}$, the order of the class of the prime ideal of $k_n$ lying above $2$ is at most $2$, and hence $|D_\Sigma(k_n)|\le 2$.
We see that $|A(k_n)|=|D_{\Sigma}(k_n)|\cdot|A^{\Sigma}(k_n)|\le 2\cdot 8 =16$ for each $n\ge0$.
Therefore, we have $|X(k_{\infty})|\le 16$.
\end{remark}

\begin{example}
By using \texttt{PARI/GP} \cite{PARI}, we obtain the following examples.

1.\ \ $k=\mathbb{Q}(\sqrt{ 3\cdot 5\cdot 29 })$: \ 
We have  $3 \cdot 5 \cdot 29 \equiv 3 \pmod 8$, \ $\mathrm{rank}_4(A(k_2))=1$, 
\[
A(k) \simeq \mathbb{Z}/2\mathbb{Z} \oplus \mathbb{Z}/2\mathbb{Z},  \quad 
A(k_1) \simeq \mathbb{Z}/4\mathbb{Z} \oplus \mathbb{Z}/2\mathbb{Z}.
\]
By Theorem \ref{MT2}, $X(k_\infty)\simeq \mathbb{Z}/4\mathbb{Z} \oplus \mathbb{Z}/2\mathbb{Z}$.\\

2.\ \ $k=\mathbb{Q}(\sqrt{ 3\cdot 5\cdot 11\cdot 19 })$: \ 
We have  $3 \cdot 5 \cdot 11 \cdot 19 \equiv 7 \pmod 8$, \ $\mathrm{rank}_4(A^{\Sigma}(k_2))=1$, 
\begin{align*}
A(k) \simeq \mathbb{Z}/2\mathbb{Z} \oplus \mathbb{Z}/2\mathbb{Z} \oplus \mathbb{Z}/2\mathbb{Z},  \quad  &D_{\Sigma}(k)\simeq \mathbb{Z}/2\mathbb{Z}, \quad   A^{\Sigma}(k) \simeq \mathbb{Z}/2\mathbb{Z} \oplus \mathbb{Z}/2\mathbb{Z},\\
A(k_1) \simeq \mathbb{Z}/4\mathbb{Z} \oplus \mathbb{Z}/2\mathbb{Z} \oplus \mathbb{Z}/2\mathbb{Z},  \quad  &D_{\Sigma}(k_1)\simeq \mathbb{Z}/2\mathbb{Z},   \quad   A^{\Sigma}(k_1) \simeq \mathbb{Z}/4\mathbb{Z} \oplus \mathbb{Z}/2\mathbb{Z}.
\end{align*}
By Corollary \ref{MT3}, $X^{\Sigma}(k_\infty)\simeq \mathbb{Z}/4\mathbb{Z} \oplus \mathbb{Z}/2\mathbb{Z}$.\\

3.\ \  $k=\mathbb{Q}(\sqrt{ 3\cdot 11\cdot 5\cdot 31 })$: \ 
We see that the prime numbers $(p_1, p_2, p_3, p_4) = (3, 11, 5, 31)$ satisfy the conditions of Theorem \ref{MT1}. 
Therefore, by Theorem \ref{MT1}, we have 
\[
X(k_\infty) \simeq \mathbb{Z}/4\mathbb{Z} \oplus \mathbb{Z}/2\mathbb{Z} \oplus \mathbb{Z}/2\mathbb{Z}, \quad A(k) \simeq \mathbb{Z}/2\mathbb{Z} \oplus \mathbb{Z}/2\mathbb{Z} \oplus \mathbb{Z}/2\mathbb{Z}.
\]
\end{example}

\section{Proof of Theorem \ref{MT1}}

In this section, we suppose that $l=2$. 
We show that the real quadratic field $k$ defined in Theorem \ref{MT1} satisfies the four conditions of Corollary \ref{MT3}.

We introduce some notation.
Let $p_1, p_2, p_3$, and $p_4$ be the prime numbers satisfying the conditions given in Theorem \ref{MT1}.
We set $m = p_1p_2p_3p_4$, $k = \mathbb{Q}(\sqrt{m})$, and $k^{\vee}= \mathbb{Q}(\sqrt{2m})$. 
For a square-free integer $d>1$, let $\varepsilon_d$ be the fundamental unit of $\mathbb{Q}(\sqrt{d})$.
For any number field $F$, let $\mathcal{O}_F$ be its ring of integers. 
We denote by $\mathfrak{l}_{F}$ a prime ideal of $F$ lying above $2$, and by $\mathfrak{p}_{i,F}$ a prime ideal of $F$ lying above $p_i$ for each $i \in \{1,2,3,4\}$. 
For an ideal $\mathfrak{a}$ of $F$, we denote by $[\mathfrak{a}]$ the ideal class containing $\mathfrak{a}$.
For a finite abelian group $A$, we put $A[2] = \{ a\in A \mid a^2=1 \}$.
For an unramified abelian extension $K/F$ of number fields and a prime ideal $\mathfrak{p}$ of $F$, we denote by $\left(\frac{K/F}{\mathfrak{p}}\right)$ the Artin symbol. 
Also, for an element $x \in \mathcal{O}_F$ prime to $\mathfrak{p}$, we denote by $\left(\frac{x}{\mathfrak{p}}\right)$ the quadratic residue symbol.
Finally, for a prime ideal $\mathfrak{p}$ of $F$ and elements $x,y \in F^{\times}$, we denote by $(x,y)_\mathfrak{p}$ the Hilbert symbol at $\mathfrak{p}$.

Since $m\equiv 3 \pmod 8$, $2$ is ramified in $k/\mathbb{Q}$, and hence $|\Sigma|=1$.
This means that  $k = \mathbb{Q}(\sqrt{m})$ satisfies condition (1) in Corollary \ref{MT3}.

\begin{lemma}\label{4_1}
$A(k)\simeq \mathbb{Z}/2\mathbb{Z} \oplus \mathbb{Z}/2\mathbb{Z} \oplus \mathbb{Z}/2\mathbb{Z}$ and  $A^{\Sigma}(k)\simeq \mathbb{Z}/2\mathbb{Z} \oplus \mathbb{Z}/2\mathbb{Z}$.
\end{lemma}

\begin{proof}
Since the narrow genus field of $k$ is $\mathbb{Q}(\sqrt{-1}, \sqrt{-p_1},\sqrt{-p_2},\sqrt{p_3},\sqrt{-p_4})$, we see that the genus field $k_G=k(\sqrt{p_1},\sqrt{p_2},\sqrt{p_3})$ is the maximal unramified elementary abelian $2$-extension of $k$, which is a subextension of $L(k)/k$.
By the Artin map, $\mathrm{Gal}(k_G/k) \simeq A(k)/2A(k) $.
This implies $\mathrm{rank}_2(A(k))=3$.
The discriminant of $k$ is $D_k= 4p_1p_2p_3p_4$.
The elements of $S_1(k)$ are listed in Table \ref{tab1}. 
For each element not belonging to $S_2(k)$, the table also gives a Kronecker symbol whose value is $-1$.
Therefore, we have $S_2(k)= \{\{1,D_{k} \}\}$.
By Proposition \ref{RRT}, $\mathrm{rank}_4(A^+(k))=0$, and thus $A(k)\simeq \mathbb{Z}/2\mathbb{Z} \oplus \mathbb{Z}/2\mathbb{Z} \oplus \mathbb{Z}/2\mathbb{Z}$.
Furthermore, since $\mathfrak{l}_k ^2 = 2\mathcal{O}_k$, the order of $[\mathfrak{l}_k]$ is at most two.
Since $\mathfrak{l}_k$ is inert in $k(\sqrt{p_3})/k$, we have $[\mathfrak{l}_k] \ne 1$.
Therefore, we have $A^{\Sigma}(k)\simeq \mathbb{Z}/2\mathbb{Z} \oplus \mathbb{Z}/2\mathbb{Z}$.
This completes the proof. 
\end{proof}

\begin{lemma}\label{4_2}
$A(k^{\vee})\simeq \mathbb{Z}/2\mathbb{Z} \oplus \mathbb{Z}/2\mathbb{Z} \oplus \mathbb{Z}/2\mathbb{Z}$.
\end{lemma}

\begin{proof}
Since the narrow genus field of $k^\vee$ is $\mathbb{Q}(\sqrt{-2},\sqrt{-p_1},\sqrt{-p_2},\sqrt{p_3},\sqrt{-p_4})$, we see that the genus field $k_{G}^\vee = k^\vee(\sqrt{2p_1},\sqrt{2p_2}, \sqrt{p_3}, \sqrt{2p_4})$ is the maximal unramified elementary abelian $2$-extension of $k^\vee$, which is a subextension of $L(k^\vee)/k^\vee$.
By the Artin map, $\mathrm{Gal}(k_{G}^\vee / k^\vee) \simeq A(k^\vee)/2A(k^\vee) $.
This implies $\mathrm{rank}_2(A(k^\vee))=3$.
The discriminant of $k^\vee$ is $D_{k^\vee}= 8p_1p_2p_3p_4$.
The elements of $S_1(k^\vee)$ are listed in Table \ref{tab2}. 
For each element not belonging to $S_2(k^\vee)$, the table also gives a Kronecker symbol whose value is $-1$.
Therefore, we have $S_2(k^\vee)= \{\{1,D_{k^\vee} \}\}$.
By Proposition \ref{RRT}, $\mathrm{rank}_4(A^+(k^\vee))=0$, and thus we have $A(k^\vee)\simeq \mathbb{Z}/2\mathbb{Z} \oplus \mathbb{Z}/2\mathbb{Z} \oplus \mathbb{Z}/2\mathbb{Z}$.
This completes the proof.
\end{proof}

\begin{table}[htpb]
\centering
\caption{The elements of $S_1(k)$ and the obstructions to $S_2(k)$}
$
\begin{array}{cc||cc} \hline
\quad\quad\quad S_1(k) \quad\quad\quad &\quad \text{Kronecker} \quad &  \quad\quad\quad S_1(k)\quad\quad\quad &\quad \text{Kronecker} \quad  \\ \hline
\{1,D_k\}              & - &\{-4p_3,-p_1p_2p_4\} & \chi_{-4p_3}(p_2) \\
\{-4,-p_1p_2p_3p_4\}   &\chi_{-4}(p_1) &\{4p_4,p_1p_2p_3\} &\chi_{p_1p_2p_3}(2) \\ 
\{-p_1,-4p_2p_3p_4\}   &\chi_{-p_1}(2) &\{p_1p_2,4p_3p_4\} &\chi_{p_1p_2}(p_3) \\
\{-p_2,-4p_1p_3p_4\}   &\chi_{-p_2}(2) &\{-p_1p_3,-4p_2p_4\} &\chi_{-p_1p_3}(p_2) \\ 
\{p_3,4p_1p_2p_4\}     &\chi_{p_3}(2)  &\{p_1p_4,4p_2p_3\} &\chi_{p_1p_4}(2) \\
\{-p_4,-4p_1p_2p_3\}   &\chi_{-p_4}(p_1) &\{-p_2p_3,-4p_1p_4\}&\chi_{-p_2p_3}(p_1) \\
\{4p_1,p_2p_3p_4\}     &\chi_{4p_1}(p_3) &\{p_2p_4,4p_1p_3\}  &\chi_{p_2p_4}(2) \\
\{4p_2,p_1p_3p_4\}     &\chi_{4p_2}(p_1) &\{-p_3p_4,-4p_1p_2\} &\chi_{-p_3p_4}(2)  \\ \hline
\end{array}
$
\label{tab1}
\end{table}

\begin{table}[htpb]
\centering
\caption{The elements of $S_1(k^\vee)$ and the obstructions to $S_2(k^\vee)$}
$
\begin{array}{cc||cc} \hline
\quad\quad\quad S_1(k^\vee)\quad\quad\quad  &\quad \text{Kronecker} \quad & \quad\quad\quad  S_1(k^\vee) \quad\quad\quad &\quad \text{Kronecker} \quad  \\ \hline
\{1,D_{k^\vee}\}              & - &\{-8p_3,-p_1p_2p_4\} &\chi_{-8p_3}(p_1) \\
\{-8,-p_1p_2p_3p_4\}   &\chi_{-8}(p_3) &\{8p_4,p_1p_2p_3\} &\chi_{8p_4}(p_1) \\ 
\{-p_1,-8p_2p_3p_4\}   &\chi_{-p_1}(2) &\{p_1p_2,8p_3p_4\} &\chi_{p_1p_2}(p_3) \\
\{-p_2,-8p_1p_3p_4\}   &\chi_{-p_2}(2) &\{-p_1p_3,-8p_2p_4\} &\chi_{-p_1p_3}(p_2)  \\ 
\{p_3,8p_1p_2p_4\}     &\chi_{p_3}(2)  &\{p_1p_4,8p_2p_3\} &\chi_{p_1p_4}(2) \\
\{-p_4,-8p_1p_2p_3\}   &\chi_{-p_4}(p_1) &\{-p_2p_3,-8p_1p_4\}&\chi_{-p_2p_3}(p_1)  \\
\{8p_1,p_2p_3p_4\}     &\chi_{8p_1}(p_2) &\{p_2p_4,8p_1p_3\} &\chi_{p_2p_4}(2) \\
\{8p_2,p_1p_3p_4\}     &\chi_{8p_2}(p_3) &\{-p_3p_4,-8p_1p_2\} &\chi_{-p_3p_4}(2) \\ \hline
\end{array}
$

\label{tab2}
\end{table}

\begin{remark}\label{4_3}
By the proofs of Lemma \ref{4_1} and \ref{4_2}, we have
\[
L(k)= k_G= k(\sqrt{p_1},\sqrt{p_2},\sqrt{p_3}), \quad  L(k^\vee)= k^\vee_G= k^\vee(\sqrt{2p_1},\sqrt{2p_2},\sqrt{p_3}).
\]
Since $2$ is inert in $\mathbb{Q}(\sqrt{p_3})/\mathbb{Q}$, we see that $k(\sqrt{p_3}) \not \subset L^{\Sigma}(k)$ and $k^{\vee}(\sqrt{p_3})\not \subset L^{\Sigma}(k^{\vee})$.
Therefore, we have  
\[
L^{\Sigma}(k)= k(\sqrt{p_1},\sqrt{p_2}), \quad  L^{\Sigma}(k^\vee)=  k^\vee(\sqrt{2p_1},\sqrt{2p_2}).
\]
We have completely determined $L^{\Sigma}(k)$ and $L^{\Sigma}(k^\vee)$.
This fact is important for the following argument.
\end{remark}

\begin{lemma}\label{4_4}
$A(k_1)\simeq \mathbb{Z}/4\mathbb{Z} \oplus \mathbb{Z}/2\mathbb{Z} \oplus \mathbb{Z}/2\mathbb{Z}$ and $A^{\Sigma}(k_1)\simeq \mathbb{Z}/4\mathbb{Z} \oplus \mathbb{Z}/2\mathbb{Z}$.
\end{lemma}

\begin{proof}
By Proposition \ref{GF} for $k_1/\mathbb{Q}_1$, we have
\[
|A(k_1)^{\mathrm{Gal}(k_1/\mathbb{Q}_1)}| = \frac{32}{(E(\mathbb{Q}_1) : E(\mathbb{Q}_1) \cap N_{k_1/\mathbb{Q}_1}(k_1^{\times}))}.
\]
We first show that $(E(\mathbb{Q}_1) : E(\mathbb{Q}_1) \cap N_{k_1/\mathbb{Q}_1}(k_1^{\times})) = 4$. 
Since $p_4$ splits in $\mathbb{Q}_1/\mathbb{Q}$, we have $(-1, m)_{\mathfrak{p}_{4, \mathbb{Q}_1}} = \left( \frac{-1}{p_4}\right) = -1$, which implies $-1\notin N_{k_1/\mathbb{Q}_1}(k_1^{\times})$. 
Since $p_1$ is inert in $\mathbb{Q}_1/\mathbb{Q}$, we have 
$(\varepsilon_2, m)_{\mathfrak{p}_{1, \mathbb{Q}_1}} = \left( \frac{ \varepsilon_2}{\mathfrak{p}_{1, \mathbb{Q}_1}}\right) = \left( \frac{ N_{\mathbb{Q}_1/\mathbb{Q}}(\varepsilon_2)}{p_1}\right)= \left( \frac{-1}{p_1}\right) =-1$, which implies $\varepsilon_2\notin N_{k_1/\mathbb{Q}_1}(k_1^{\times})$.
Furthermore, $(-\varepsilon_2, m)_{\mathfrak{p}_{1, \mathbb{Q}_1}} =( -1, m)_{\mathfrak{p}_{1, \mathbb{Q}_1}}( \varepsilon_2, m)_{\mathfrak{p}_{1, \mathbb{Q}_1}}=  \left( \frac{-1}{\mathfrak{p}_{1, \mathbb{Q}_1}}\right)\left(\frac{ \varepsilon_2}{\mathfrak{p}_{1, \mathbb{Q}_1}}\right) = 1\cdot (-1)=-1$, which implies $-\varepsilon_2\notin N_{k_1/\mathbb{Q}_1}(k_1^{\times})$.
It follows that $E(\mathbb{Q}_1) \cap N_{k_1/\mathbb{Q}_1}(k_1^{\times})= E(\mathbb{Q}_1)^2$. 
Therefore, $(E(\mathbb{Q}_1) : E(\mathbb{Q}_1) \cap N_{k_1/\mathbb{Q}_1}(k_1^{\times})) = (E(\mathbb{Q}_1) : E(\mathbb{Q}_1)^2) = 4$.
Since $|A(\mathbb{Q}_1)|=1$, we have $|A(k_1)^{\mathrm{Gal}(k_1/\mathbb{Q}_1)}|=|A(k_1)/2A(k_1)|= 2^{\mathrm{rank}_2(A(k_1))}$, and hence $\mathrm{rank}_2(A(k_1))= 3$. 
By Lemma \ref{4_1} and \ref{4_2}, we see that $|A(k)|=|A(k^\vee)|=8$.
By Proposition \ref{KKT} for $k_1/\mathbb{Q}$, we have
\[
|A(k_1)|= \frac{1}{4}\cdot Q(k_1)\cdot |A(k)|\cdot |A(k^{\vee})|\cdot |A(\mathbb{Q}_1)| = \frac{1}{4}\cdot Q(k_1) \cdot 8\cdot 8 \cdot 1 = 16\cdot Q(k_1).
\]
If a system of fundamental units of $k_1$ is $\{ \varepsilon_2, \varepsilon_m, \varepsilon_{2m} \}$, then $Q(k_1)=1$, which implies $|A(k_1)|=16$.
Under this assumption, since $\mathrm{rank}_2(A(k_1))= 3$, it follows that $A(k_1)\simeq \mathbb{Z}/4\mathbb{Z} \oplus \mathbb{Z}/2\mathbb{Z} \oplus \mathbb{Z}/2\mathbb{Z}$.
We now prove that $\{ \varepsilon_2, \varepsilon_m, \varepsilon_{2m} \}$ is a system of fundamental units of $k_1$.
By the proof of Lemma \ref{4_1}, the narrow genus field of $k$ is different from the genus field $k_G$. 
Thus, we have $A(k) \neq A^+(k)$. 
Similarly, by the proof of Lemma \ref{4_2}, we have $A(k^\vee) \neq A^+(k^\vee)$. 
Therefore, we have $N_{k/\mathbb{Q}}(\varepsilon_m) = N_{k^{\vee}/\mathbb{Q}}(\varepsilon_{2m}) = 1$.
On the other hand, we easily see that $N_{\mathbb{Q}_1/\mathbb{Q}}(\varepsilon_2)=-1$.
By Proposition \ref{KKT}, it suffices to show that $\sqrt{\varepsilon_m}, \sqrt{\varepsilon_{2m}}, \sqrt{\varepsilon_m\varepsilon_{2m}} \notin k_1$.
We consider two cases, depending on the value of $\left( \frac{p_2}{p_4}\right)$.
We also recall that $L(k)= k(\sqrt{p_1}, \sqrt{p_2}, \sqrt{p_3})$ and $L(k^{\vee})= k^{\vee}(\sqrt{2p_1}, \sqrt{2p_2}, \sqrt{p_3})$ (see Remark \ref{4_3}).

\noindent \textbf{Case 1:} 
We suppose that $\left( \frac{p_2}{p_4}\right)=-1$.
First, we show that $\sqrt{\varepsilon_m} \notin k_1$. 
Since $\mathfrak{p}_{2, k}$ splits completely in $L(k)/k$, it is a principal ideal in $k$.	
Hence, there exists $\alpha \in k^{\times}$ such that $\mathfrak{p}_{2, k} = \alpha \mathcal{O}_k$, meaning $p_2= \varepsilon_m^z\alpha^2$ for some $z\in \mathbb{Z}$. 
If $z$ is even, then $\sqrt{p_2} = \pm \varepsilon_m^{\frac{z}{2}}\alpha \in k$.
This is a contradiction.
Thus, $z$ must be odd, giving $p_2=\varepsilon_m \beta^2$ for some $\beta \in k^{\times}$. 
This implies $k(\sqrt{p_2})= k(\sqrt{\varepsilon_m})$, and since $\sqrt{p_2} \notin k_1$, we obtain $\sqrt{\varepsilon_m} \notin k_1$.
Next, we show that $\sqrt{\varepsilon_{2m}} \notin k_1$. 
Put $K^{\vee}_i= k^{\vee}(\sqrt{2p_i})$ for each $i \in \{ 1,2\}$ and $K^\vee_3= k^\vee(\sqrt{p_3})$.
Then we see that $L(k^\vee)=K^{\vee}_1K^{\vee}_2K^{\vee}_3$.
Table \ref{tab3} gives the values used to determine the splitting behavior of $\mathfrak{l}_{k^\vee}$ and $\mathfrak{p}_{1,k^{\vee}}$ in $K^{\vee}_i/k^{\vee}$.
The first row gives the values modulo $8$, and the second row gives the corresponding quadratic residue symbols.
In the table, ``sp'' means that the prime ideal splits in $K_i^\vee/k^\vee$, whereas ``in'' means that it is inert in $K_i^\vee/k^\vee$.
In the cases listed in the table, an entry congruent to $1\pmod 8$, or a quadratic residue symbol equal to $1$, means that a prime splits in $K^{\vee}_i/k^{\vee}$, whereas an entry congruent to $5\pmod 8$, or a quadratic residue symbol equal to $-1$, means that it is inert in $K^{\vee}_i/k^{\vee}$.
Hence,
\[
\left( \frac{L(k^\vee)/k^\vee}{\mathfrak{l}_{k^\vee}}\right) = \left( \frac{L(k^\vee)/k^\vee}{\mathfrak{p}_{1, k^\vee}} \right),
\]
which implies $[\mathfrak{l}_{k^\vee}]=[\mathfrak{p}_{1, k^\vee}]$. 
Thus, there exists $\alpha' \in (k^\vee)^\times$ such that $\mathfrak{l}_{k^\vee} = (\alpha' \mathcal{O}_{k^\vee})\mathfrak{p}_{1, k^\vee}$, meaning $2= \varepsilon_{2m}^{z'}(\alpha')^2p_1$ for some $z'\in \mathbb{Z}$. 
If $z'$ is even, then $\sqrt{2} = \pm \varepsilon_{2m}^{\frac{z'}{2}}\alpha' \sqrt{p_1}$, which implies $\sqrt{p_1} \in k_1$.
This is a contradiction.
Thus, $z'$ must be odd, giving $2=\varepsilon_{2m} (\beta')^2p_1$ for some $\beta' \in (k^\vee)^{\times}$. 
This implies $k_1(\sqrt{p_1}) = k_1(\sqrt{\varepsilon_{2m}})$, and since $\sqrt{p_1} \notin k_1$, we obtain $\sqrt{\varepsilon_{2m}} \notin k_1$.
Lastly, we have  $k_1(\sqrt{p_1}) = k_1(\sqrt{\varepsilon_{2m}})$ and $k_1(\sqrt{p_2}) = k_1(\sqrt{\varepsilon_{m}})$.
Since $\sqrt{p_1p_2}\notin k_1$, we have $k_1(\sqrt{p_1})\neq k_1(\sqrt{p_2})$.
Therefore, $k_1(\sqrt{\varepsilon_m}) \neq k_1(\sqrt{\varepsilon_{2m}})$.
It follows that
$\sqrt{\varepsilon_m\varepsilon_{2m}}\notin k_1$.

\noindent \textbf{Case 2:} 
We suppose  that $\left( \frac{p_2}{p_4}\right)=1$.
First, we show that $\sqrt{\varepsilon_m} \notin k_1$. 
Put $K_i= k(\sqrt{p_i})$ for each $i \in \{ 1,2,3 \}$.
Then we see that $L(k)=K_1K_2K_3$.
As in Case 1, Table \ref{tab4} gives the values used to determine the splitting behavior of $\mathfrak{l}_k, \mathfrak{p}_{1,k}$ and $\mathfrak{p}_{2,k}$ in $K_i/k$.
Hence,
\[
\left( \frac{L(k)/k}{\mathfrak{l}_{k}}\right) = \left( \frac{L(k)/k}{\mathfrak{p}_{1, k}} \right) \left( \frac{L(k)/k}{\mathfrak{p}_{2, k}} \right)=\left( \frac{L(k)/k}{\mathfrak{p}_{1, k}\mathfrak{p}_{2, k}} \right),
\]
which implies $[\mathfrak{l}_{k}]=[\mathfrak{p}_{1, k}\mathfrak{p}_{2, k}]$. 
Thus, there exists $\alpha \in k^\times$ such that $\mathfrak{l}_{k} = (\alpha \mathcal{O}_{k})\mathfrak{p}_{1, k}\mathfrak{p}_{2, k}$, meaning $2= \varepsilon_m^{z}\alpha^2p_1p_2$ for some $z \in \mathbb{Z}$.
If $z$ is even, then $\sqrt{2} = \pm \varepsilon_{m}^{\frac{z}{2}}\alpha \sqrt{p_1p_2}$, which implies $\sqrt{p_1p_2} \in k_1$.
This is a contradiction.
Thus, $z$ must be odd, giving $2=\varepsilon_{m} \beta^2 p_1p_2$ for some $\beta \in k^{\times}$.
This implies $k_1(\sqrt{p_1p_2}) = k_1(\sqrt{\varepsilon_m})$, and since $\sqrt{p_1p_2} \notin  k_1$, we obtain $\sqrt{\varepsilon_{m}} \notin k_1$.
The argument used in Case 1 to prove $\sqrt{\varepsilon_{2m}} \notin k_1$ does not depend on the value of $\left(\frac{p_2}{p_4}\right)$.
Hence, $\sqrt{\varepsilon_{2m}}\notin k_1$ and $k_1(\sqrt{p_1}) = k_1(\sqrt{\varepsilon_{2m}})$.
Lastly, we have  $k_1(\sqrt{p_1p_2}) = k_1(\sqrt{\varepsilon_{m}})$ and $k_1(\sqrt{p_1}) = k_1(\sqrt{\varepsilon_{2m}})$.
Since $\sqrt{p_2}\notin k_1$, we have $k_1(\sqrt{p_1p_2})\neq k_1(\sqrt{p_1})$.
Therefore, $k_1(\sqrt{\varepsilon_{m}}) \neq k_1(\sqrt{\varepsilon_{2m}})$.
It follows that
$\sqrt{\varepsilon_m\varepsilon_{2m}}\notin k_1$.

\noindent We conclude  that $A(k_1)\simeq \mathbb{Z}/4\mathbb{Z} \oplus \mathbb{Z}/2\mathbb{Z} \oplus \mathbb{Z}/2\mathbb{Z}$.
Finally, we determine the structure of $A^{\Sigma}(k_1)$.
Since $\mathfrak{l}_{k_1}^2=\sqrt{2}\mathcal{O}_{k_1}$, the order of $[\mathfrak{l}_{k_1}]$ is at most two.
The extension $k_1(\sqrt{p_3})/k_1$ is an unramified quadratic extension (see Remark \ref{4_3}).
Since $\mathfrak{l}_{k_1}$ is inert in $k_1(\sqrt{p_3})/k_1$, the restriction of the Artin symbol $\left(\frac{k_1L(k)/k_1}{\mathfrak{l}_{k_1}} \right)$ to $k_1(\sqrt{p_3})$ is nontrivial. 
Hence, $\left(\frac{k_1L(k)/k_1}{\mathfrak{l}_{k_1}}\right)$ is nontrivial.
On the other hand, the Artin map $A(k_1)\rightarrow \mathrm{Gal}(k_1L(k)/k_1)$ sends every element of $2A(k_1)$ to the identity. 
Indeed, since $\mathrm{Gal}(k_1L(k)/k_1) \simeq A(k) \simeq \mathbb{Z}/2\mathbb{Z} \oplus \mathbb{Z}/2\mathbb{Z} \oplus \mathbb{Z}/2\mathbb{Z}$, this Galois group has exponent $2$.
Thus, the image of the square of any ideal class is trivial. 
Consequently, $[\mathfrak{l}_{k_1}]\notin 2A(k_1)$.
This implies that $D_{\Sigma}(k_1) \not\subset 2A(k_1)$. 
Therefore, $A^{\Sigma}(k_1) = A(k_1)/D_{\Sigma}(k_1) \simeq \mathbb{Z}/4\mathbb{Z} \oplus \mathbb{Z}/2\mathbb{Z}$.
This completes the proof.
\end{proof}

\begin{table}[htbp]
\centering
\caption{The splitting behavior in $K_i^\vee/k^\vee$}
$
\begin{array}{c|ccc} \hline
&\qquad\qquad K^\vee_1 \qquad\qquad & \qquad\qquad K^\vee_2\qquad\qquad & \qquad\qquad K^\vee_3\qquad\qquad \\ \hline
\mathfrak{l}_{k^\vee}   \,\, & \,\, p_2p_3p_4\equiv 1\quad\text{(sp)} & p_1p_3p_4\equiv 1\quad\text{(sp)} &p_3\equiv 5 \quad\text{(in)} \\
\mathfrak{p}_{1,k^\vee} \,\,&\,\, \big(\frac{p_2p_3p_4}{p_1}\big)=1 \quad\text{(sp)}&\big(\frac{2p_2}{p_1}\big)=1 \quad\text{(sp)}&\big(\frac{p_3}{p_1}\big)=-1 \!\!\quad\text{(in)} \\ \hline
\end{array}
$
\label{tab3}
\end{table}

\begin{table}[htbp]
\centering
\caption{The splitting behavior in $K_i/k$}
$
\begin{array}{c|ccc} \hline
&\qquad\qquad K_1 \qquad\qquad & \qquad\qquad K_2\qquad\qquad & \qquad\qquad K_3\qquad\qquad \\ \hline
\mathfrak{l}_k   \,\, & \,\, p_2p_3p_4\equiv 1\quad\text{(sp)} & p_1p_3p_4\equiv 1\quad\text{(sp)} &p_3\equiv 5 \quad\text{(in)} \\
\mathfrak{p}_{1,k} \,\,&\,\, \big(\frac{p_2p_3p_4}{p_1}\big)=1 \quad\text{(sp)}&\big(\frac{p_2}{p_1}\big)=-1 \quad\text{(in)}&\big(\frac{p_3}{p_1}\big)=-1 \quad\text{(in)} \\
\mathfrak{p}_{2,k} \,\, &\,\, \big(\frac{p_1}{p_2}\big)=1 \quad\text{(sp)}&\big(\frac{p_1p_3p_4}{p_2}\big)=-1 \quad\text{(in)}&\big(\frac{p_3}{p_2}\big) =1 \quad\text{(sp)}\\ \hline
\end{array}
$

\label{tab4}
\end{table}

\begin{lemma}\label{4_6}
$\mathrm{rank}_4(A^{\Sigma}(k_2))=1$.
\end{lemma}

\begin{proof}
We have $\mathrm{rank}_2(A^{\Sigma}(k_2))=2$ by Lemmas \ref{4_1}, \ref{4_4} and  Proposition \ref{FT1}.
The norm map $N_{k_2/k_1}: A(k_2) \to A(k_1)$ induces a surjective homomorphism  $A^{\Sigma}(k_2) \twoheadrightarrow A^{\Sigma}(k_1)$, and hence we obtain $2A^{\Sigma}(k_2)/4A^{\Sigma}(k_2)\twoheadrightarrow 2A^{\Sigma}(k_1)/4A^{\Sigma}(k_1)$.
Since $\mathrm{rank}_4(A^{\Sigma}(k_1)) = 1$ by Lemma \ref{4_4}, we have $\mathrm{rank}_4(A^{\Sigma}(k_2))\geq 1$.
We suppose that $\mathrm{rank}_4(A^{\Sigma}(k_2))=2$. 
Then we have $A^{\Sigma}(k_2)[2] \subset 2A^{\Sigma}(k_2)$. 
Let $a$ denote the image of the ideal class $[\mathfrak{p}_{4,k_2}]$ in $A^{\Sigma}(k_2)$.
We see that $p_4$ splits completely in $\mathbb{Q}_2/\mathbb{Q}$ and $\mathfrak{p}_{4,\mathbb{Q}_2}$ is ramified in $k_2/\mathbb{Q}_2$. 
Since $|A(\mathbb{Q}_2)|=1$, the order of $a$ is at most two, which implies $a \in A^{\Sigma}(k_2)[2]$. 
Since $k(\sqrt{p_1}) \cap k_2 = k$, it follows that $k_2(\sqrt{p_1})/k_2$ is a $\Sigma$-decomposed unramified quadratic extension.
Since $\mathfrak{p}_{4,k_2}$ is inert in $k_2(\sqrt{p_1})/k_2$, the restriction of the Artin symbol $\left(\frac{k_2L^{\Sigma}(k)/k_2}{\mathfrak{p}_{4,k_2}}\right)$ to $k_2(\sqrt{p_1})$ is nontrivial. 
Hence, $\left(\frac{k_2L^{\Sigma}(k)/k_2}{\mathfrak{p}_{4,k_2}}\right)$ is nontrivial. 
On the other hand, the Artin map $A^{\Sigma}(k_2)\rightarrow \mathrm{Gal}(k_2L^{\Sigma}(k)/k_2)$ sends every element of $2A^{\Sigma}(k_2)$ to the identity. 
Indeed, since $\mathrm{Gal}(k_2L^{\Sigma}(k)/k_2) \simeq A^{\Sigma}(k) \simeq \mathbb{Z}/2\mathbb{Z}\oplus\mathbb{Z}/2\mathbb{Z}$, this Galois group has exponent $2$. 
Thus, the image of the square of any ideal class is trivial.
Therefore, we have $a \notin 2A^{\Sigma}(k_2)$. 
This is a contradiction.
Therefore, we have $\mathrm{rank}_4(A^{\Sigma}(k_2))=1$.
This completes the proof.
\end{proof}

\begin{proof}[Proof of Theorem \ref{MT1}]
Recall that  $2$ is ramified in $k/\mathbb{Q}$. 
Furthermore, by Lemmas \ref{4_1}, \ref{4_4}, and \ref{4_6}, all the conditions of Corollary \ref{MT3} are satisfied, and hence we have $X^\Sigma(k_\infty) \simeq \mathbb Z/4\mathbb Z\oplus\mathbb Z/2\mathbb Z$.
The natural map $X^\Sigma(k_\infty)\to \mathrm{Gal}(L^{\Sigma}(k_2)/k_2)$ is surjective and $\mathrm{Gal}(L^{\Sigma}(k_2)/k_2)\simeq A^{\Sigma}(k_2)$. 
Hence, we have $|A^{\Sigma}(k_2)|\le 8$.
On the other hand, the norm map $N_{k_2/k_1}: A(k_2) \to A(k_1)$ induces a surjective homomorphism $A^{\Sigma}(k_2) \twoheadrightarrow A^{\Sigma}(k_1)$, and $ A^\Sigma(k_1) \simeq \mathbb Z/4\mathbb Z\oplus\mathbb Z/2\mathbb Z$.
Hence, we have $|A^{\Sigma}(k_2)|\ge 8$.
It follows that $|A^{\Sigma}(k_2)|=8$.
Therefore, the above surjective homomorphism is an isomorphism, and hence $A^{\Sigma}(k_2) \simeq A^{\Sigma}(k_1) \simeq \mathbb Z/4\mathbb Z\oplus\mathbb Z/2\mathbb Z$.
Finally, we determine the structure of $A(k_2)$. 
Recall that the second layer of the cyclotomic $\mathbb{Z}_2$-extension of $\mathbb{Q}$ is $\mathbb{Q}_2 = \mathbb{Q}(\sqrt{2+\sqrt{2}})$. 
Since the norm map $N_{k_2/k_1} : A(k_2) \to A(k_1)$ is surjective, we have $|A(k_2)| \ge |A(k_1)| = 16$. 
On the other hand, note that $\mathfrak{l}_{k_2}^2 = \sqrt{2+\sqrt{2}}\mathcal{O}_{k_2}$. 
This implies that $[\mathfrak{l}_{k_2}]^2 = 1$ in $A(k_2)$, and hence the order of $D_{\Sigma}(k_2) = \langle [\mathfrak{l}_{k_2}] \rangle$ is at most $2$. 
Since $A^{\Sigma}(k_2) = A(k_2)/D_{\Sigma}(k_2) \simeq \mathbb{Z}/4\mathbb{Z} \oplus \mathbb{Z}/2\mathbb{Z}$, we have $|A(k_2)| \le 8 \cdot 2 = 16$. 
It follows that $|A(k_2)|=16$.
Therefore, the norm map is an isomorphism, and hence $A(k_2) \simeq A(k_1) \simeq \mathbb Z/4\mathbb Z\oplus\mathbb Z/2\mathbb Z\oplus\mathbb Z/2\mathbb Z$.
By Proposition \ref{FT1} (1), $A(k_n)\simeq \mathbb Z/4\mathbb Z\oplus\mathbb Z/2\mathbb Z\oplus\mathbb Z/2\mathbb Z $ for all $n\geq 1 $.
Consequently, $X(k_\infty) \simeq \mathbb Z/4\mathbb Z \oplus \mathbb Z/2\mathbb Z \oplus \mathbb Z/2\mathbb Z$.
In particular, $\lambda_2(k)=\mu_2(k)=0, \ \nu_2(k)=4$.
This completes the proof of Theorem \ref{MT1}.
\end{proof}

\bigskip
\begin{acknowledgements}
I am especially grateful to my advisor, Prof.\ Yasushi Mizusawa, for his guidance in seminars and throughout the preparation of this paper. 
I would also like to thank Prof.\ Thomas Geisser for carefully reading the manuscript and providing valuable comments. 
\end{acknowledgements}

\end{document}